\documentclass[11pt]{article}
\usepackage[margin=1in]{geometry}
\usepackage{amsmath,amssymb,amsthm,mathtools}
\usepackage{mathrsfs,mathabx}  
\usepackage{microtype}
\usepackage{enumitem,tikz}
\usepackage[colorlinks=true,linkcolor=blue,citecolor=blue,urlcolor=blue]{hyperref}

\newtheorem{theorem}{Theorem}[section]
\newtheorem{lemma}[theorem]{Lemma}

\newtheorem{definition}[theorem]{Definition}
\newtheorem{fact}[theorem]{Fact} 
\newtheorem{conjecture}[theorem]{Conjecture}
\newtheorem{observation}[theorem]{Observation}
\newtheorem{remark}[theorem]{Remark}

\usepackage[comma,sort, numbers]{natbib}

\title{ Almost Empty Monochromatic Triangles With Many Colors }

\author{Bhaswar B. Bhattacharya\thanks{Department of Statistics and Data Science, University of Pennsylvania and National University of Singapore \texttt{bhaswar@wharton.upenn.edu}} \and 
Sandip Das\thanks{Advanced Computing and Microelectronics Unit, Indian Statistical Institute, Kolkata, \texttt{sandip.das.69@gmail.com}}
\and 
Sk Samim Islam\thanks{Advanced Computing and Microelectronics Unit, Indian Statistical Institute, Kolkata, \texttt{samimislam08@gmail.com}}
\and 
Aashirwad Mohapatra\thanks{Advanced Computing and Microelectronics Unit, Indian Statistical Institute, Kolkata, \texttt{aashirwad\_r@isical.ac.in}} 
\and 
Saumya Sen\thanks{Advanced Computing and Microelectronics Unit, Indian Statistical Institute, Kolkata, \texttt{saumyasen72@gmail.com}}
}

\date{}

\begin{document}
\maketitle

\begin{abstract} 
Given integers $c\geq 2$ and $s\geq 0$, let $\mathsf{M}_3(c,s)$ denote the least
integer such that every set of at least $\mathsf{M}_3(c,s)$ points in the
plane, no three on a line, colored with $c$ colors, contains a monochromatic
triangle with at most $s$ interior points.  Further, let $\lambda_3(c)$ be the
least integer such that $\mathsf{M}_3(c,\lambda_3(c))<\infty$. \citet{colorempty} proved that, for every $c\geq 2$,
$$\left\lfloor\frac{c-1}{2}\right\rfloor \leq \lambda_3(c)\leq c-2.$$
Later, \citet{cravioto2019almost} improved the upper bound to $c-3$, for
$c\geq 4$.  In this paper, we refine their argument
to obtain the following asymptotic improvement: 
$$\lambda_3(c) \leq c-\sqrt{c\log c}+o (\sqrt{c\log c} ),$$ 
for all sufficiently large $c$. We also show that every $c$-coloring of a sufficiently large Horton set contains a monochromatic triangle with at most $\lfloor \frac{c-1}{2} \rfloor$ interior points. This shows that the aforementioned lower bound on $\lambda_3(c)$ is sharp within the class of Horton sets.  We conclude with a conjecture on the large-color asymptotics of $\lambda_3(c)$.    
\end{abstract}

\section{Introduction}

The celebrated Erd\H{o}s--Szekeres theorem \cite{erdos1935combinatorial,erdHos1960some} asserts that, for every integer $r\geq 3$, there exists a smallest integer $\mathsf{ES}(r)$ such that every set of at least $\mathsf{ES}(r)$ points in the plane in general position (that is, with no three points collinear) contains $r$ points in convex position, equivalently, $r$ points that form the vertex set of a convex $r$-gon.  This theorem is a cornerstone of combinatorial geometry and
geometric Ramsey theory, and has inspired many refinements and extensions (see, for example,
\cite{morris2000erdos,barany2001problems,toth2005erdos,holmsen2020two}
and the references therein). One particular strengthening, introduced by Erd\H{o}s~\cite{erdos1978some}, asks for an $r$-\emph{hole}, that is, a convex $r$-gon whose interior contains no other point of the set.  Let $\mathsf{H}(r)$ denote the minimum integer, if it exists, such that every set of $\mathsf{H}(r)$ points in general position contains an $r$-hole.  It is known that
$\mathsf{H}(3)=3$, $\mathsf{H}(4)=5$, and $\mathsf{H}(5)=10$ \cite{harborth1978konvexe}.  \citet{horton1983sets} showed that $\mathsf{H}(r)$ does not exist for $r\ge 7$. The finiteness of $\mathsf{H}(6)$ remained open for many years before being established independently by \citet{gerken2008empty} and \citet{nicolas2007empty}. Recently, \citet{heule2024happy} determined the exact value $\mathsf{H}(6)=30$, thereby resolving the longstanding 6-hole problem.

The non-existence of $\mathsf{H}(r)$ for $r\geq 7$, motivated the study of different relaxations.  One such direction concerns \emph{almost empty} convex polygons, whose interiors contain only a bounded number of points \cite{nyklova2003almost,koshelev2011computer,huemer2022weighted}, which is also closely related to the study of islands \cite{bautista2011computing,balko2022holes}.  Other variants impose arithmetic conditions on the number of interior points \cite{bialostocki1991some}, or introduce vertex colors (see \cite{devillers2003chromatic,fabila2021empty,diaz2021note,arevalo2022rainbow,kano2021discrete,bereg2015balanced} and references therein). In the colored setting, a polygon is \emph{monochromatic} if all of its vertices have the same color.  For two colors, \citet{grima2009some} proved that every set of $9$ points contains a monochromatic $3$-hole. Quantitatively, \citet{aichholzer2009empty} showed that every sufficiently large $2$-colored $n$-point set contains $\Omega(n^{\frac{5}{4}})$ empty monochromatic triangles, improved to $\Omega(n^{\frac{4}{3}})$ by \citet{pach2013monochromatic}.  In contrast, \citet{devillers2003chromatic} constructed arbitrarily large $3$-colored sets with no monochromatic $3$-hole and arbitrarily large $2$-colored sets with no monochromatic $5$-hole. They also conjectured that every sufficiently large $2$-colored point set contains a monochromatic $4$-hole.  Although this conjecture remains open, \citet{aichholzer2010large} established a relaxed version in which the 4-hole is not required to be convex.

The non-existence of empty monochromatic triangles in point sets with three or more colors, prompted \citet{colorempty} to initiate the study of almost empty monochromatic triangles. 
Formally, for integers $c \geq 2$ and $s \geq 0$, define $\mathsf{M}_3(c,s)$ to be the least integer such that any set of at least $\mathsf{M}_3(c,s)$  points in the plane in general position, colored with $c$ colors, contains a monochromatic triangle with at most $s$ interior points. Moreover, denote by $\lambda_3(c)$ the least integer such that $\mathsf{M}_3(c, \lambda_3(c)) < \infty$.  \citet{colorempty} proved that, for any $c \geq 2$, 
\begin{align}\label{eq:almostempty}
\left\lfloor \frac{c-1}{2} \right\rfloor \leq \lambda_3(c) \leq c-2. 
\end{align} 
Specifically, for $c=3$, this implies $\lambda_3(3)=1$, that is, every sufficiently large 3-colored point set contains a monochromatic triangle with at most one interior point (see also the recent paper \cite{counting} for a quantitative superlinear bound). The upper bound in \eqref{eq:almostempty} was later improved by \citet{cravioto2019almost} to $c-3$, for $c \geq 4$. In particular, for $c = 4$, this implies $\lambda_3(4)=1$, which settled a conjecture in \cite{colorempty}. The upper bound in \cite{cravioto2019almost} together with the lower bound in \eqref{eq:almostempty} also show that $\lambda_3(5)=2$.

The above results combined determine the exact values of $\lambda_3(c)$ for small values of $c$, specifically for $c\leq 5$, and provide linear (in terms of $c$) bounds with constant factor corrections, for general $c \geq 6$. A complementary question is to understand the behavior of $\lambda_3(c)$ as the number of colors becomes large. In this paper, we take a first step in this direction by proving the following result, which improves the additive correction of $-3$ in the upper bound of \citet{cravioto2019almost} to order $-\sqrt{c\log c}$, for $c$ sufficiently large. Throughout, all logarithms are natural, unless a base is explicitly indicated.

\begin{theorem}\label{thm:triangle}
For $c$ sufficiently large, 
$$\lambda_3(c)    \leq c- \left\lfloor \sqrt{c \log c-4 c \log\log c }\right\rfloor . $$
\end{theorem}


The proof of Theorem~\ref{thm:triangle} is given in Section~\ref{sec:trianglepf}. It follows the same high-level strategy as \citet{cravioto2019almost}: choose a most frequent color class, triangulate it, use points of the other colors to block the resulting monochromatic triangular faces, and then repeat the argument inside its convex hull. To obtain the improved bound, we introduce two new ingredients into this general strategy: First, at each step of the iteration, instead of choosing
a most frequent color among all colors present, as in \cite{cravioto2019almost}, we choose a most frequent color among the unused colors, which prevents repetitions. Second, we allow the number of recursive steps to grow with the number of colors $c$. This requires controlling the cumulative contribution of all previously selected colors throughout a long iteration. By carefully balancing the number of recursive steps against this accumulated error, we obtain the $\sqrt{c\log c}$ improvement,
whose scale can be understood heuristically through a birthday-paradox interpretation of the
products arising in the recursion (see Remark~\ref{remark:product}).

\begin{remark} 
{\em It is possible to keep track of the constants throughout the proof of Theorem~\ref{thm:triangle} and make the asymptotic argument completely explicit. However, the resulting threshold for $c$ would be far from optimal. Hence, to avoid notational clutter, we have chosen to present the result in asymptotic form. Similarly, the coefficient $4$ in the $4c\log\log c$ term is chosen for convenience rather than optimality and may be reduced by sharpening the product estimates and the control of
the recursive errors. A more refined analysis may also improve the coefficient of the leading $\sqrt{c\log c}$ correction, but we do not pursue such refinements here. The present techniques, however, appear to encounter a natural barrier at the $\sqrt{c\log c}$ scale, and obtaining an additive correction of larger order would likely require new ideas. } 
\end{remark}

Next, we consider the lower bound on $\lambda_3(c)$. \citet{colorempty} established the lower bound in \eqref{eq:almostempty} by constructing arbitrarily large $c$-colored Horton sets (see Definition~\ref{defn:H}) in which every monochromatic triangle contains at least $\left\lfloor\frac{c-1}{2}\right\rfloor$ interior points. The next result shows that this lower bound is tight within the class of Horton sets. 

\begin{theorem}\label{thm:colortriangle}
For every integer $c\geq 2$, there exists an integer $n_c$ such that any $c$-coloring of a Horton set of size at least $n_c$, contains a monochromatic triangle with at most
$\left\lfloor\frac{c-1}{2}\right\rfloor$ interior points. 
\end{theorem}

The proof of Theorem~\ref{thm:colortriangle} is given in Section~\ref{sec:colortrianglepf}. The proof proceeds by induction on the number of colors and exploits the binary structure of the Horton set to descend to an appropriate level of its recursive decomposition, where a parity argument yields the desired bound on the number of interior points. Since Horton sets form one of the principal sources of extremal constructions in these problems, Theorem~\ref{thm:colortriangle} suggests that the true value of $\lambda_3(c)$ may lie closer to the lower bound. Motivated by this, we make the following conjecture:

\begin{conjecture}\label{conjecture:interiorpoints} 
$\lim_{c \rightarrow \infty} \frac{\lambda_3(c)}{c} = \frac{1}{2}$.
\end{conjecture}

Note that \eqref{eq:almostempty} implies that $\lim_{c\to\infty}\frac{\lambda_3(c)}{c}\in [\frac{1}{2},1]$, if the limit exists. Theorem~\ref{thm:triangle} improves the second-order term in the upper bound, but does not change its leading constant.  Thus, any improvement in the asymptotic leading constant would already be a significant strengthening of the current bounds.

\section{Preliminaries}

Let $S$ be a set of points in the plane in general position, that is, no three points of the set $S$ are on a line. We denote the convex hull of $S$ by $CH(S)$ and the interior of $CH(S)$ by $\mathrm{int}(CH(S))$. The set of boundary vertices of $CH(S)$ is denoted by $V(CH(S))$. For an integer $c \geq 1$, a $c$-coloring of $S$ is a function $\phi: S \rightarrow [c]:= \{1, 2, \ldots, c\}$, where $\phi(s) = a$ means that the point $s \in S$ is assigned color $a \in [c]$. We will write $S_a$ to denote the subset of points of color $a \in [c]$, that is, 
\begin{align*} 
S_a := \{s \in S: \phi(s) = a\} .  
\end{align*}
Clearly, $S_1, S_2, \ldots, S_c$ are disjoint sets and $\sum_{a=1}^c |S_a| = |S|$.  

The following standard fact about triangulation of point sets will be an important ingredient in the proofs.

\begin{fact}\label{ft:triangle}
Let $S$ be a set of points in the plane in general position. Then every triangulation of $S$
has $2|S|-|V(CH(S))|-2$ triangles.  
\end{fact}

Throughout the paper we will use the following asymptotic notation: For two nonnegative sequences $a_n$ and $b_n$ we will write $a_n = O(b_n)$, if for all $n$ large enough, $a_n \leq C _1 b_n$, for some constant $C_1 > 0$. Also, $a_n = o(b_n)$ will mean $a_n/b_n \rightarrow 0$ and $a_n \sim b_n$ will mean $a_n/ b_n \rightarrow 1$, as $n \rightarrow \infty$.

\section{Proof of Theorem \ref{thm:triangle}}
\label{sec:trianglepf}

For each large $c$, define 
\begin{align}\label{eq:interiorpoints}
 t_c := \left\lfloor \sqrt{c \log c-4 c \log\log c }\right\rfloor.
\end{align}
We will prove Theorem \ref{thm:triangle} by contradiction. For this, if possible, suppose $\mathsf{M}_3(c, c -  t_c ) =\infty$.  Then for every $N \geq 1$, there exists a $c$-colored point set $S$ in general position, with $|S| \geq N$, such that every monochromatic triangle in $S$ has at least 
\begin{align}\label{eq:sc}
s_c := c -  t_c +1 
\end{align}
interior points. We will refer to a set with this property as a {\it bad} set. 
%
%
Now, define 
\begin{align} 
R=  t_c + \varepsilon_c , \quad \text{ where } \varepsilon_c  = \left \lfloor \sqrt{\frac{c}{\log c}} \right\rfloor. 
\label{eq:Rlayers}
\end{align}
Clearly, $R<c$ for all sufficiently large $c$. Then, given a bad set $S$, construct nested sets 
\begin{align}\label{eq:sequence}
S^{(0)} \supseteq S^{(1)} \supseteq \cdots \supseteq S^{(R)}
\end{align} 
and relabel the colors as they are selected as follows: To begin with, let $S^{(0)}=S$. For $r \geq 1$, having selected $S^{(0)}, \ldots, S^{(r-1)}$, choose a color $r \in [c]$ that is not chosen earlier and whose class in $S^{(r-1)}$ has maximum cardinality among the unused colors. Then, define (see Figure \ref{fig:setcolors} for an illustration): 
$$S^{(r)}=S^{(r-1)}\cap \mathrm{int}(CH(S^{(r-1)}_{r})).$$  
Recall that $S^{(r-1)}_{r}$ is the subset of $S^{(r-1)}$ of color $r$. Hence, $S^{(r)}$ is the set of points in $S^{(r-1)}$ that lie in the interior of the convex hull of $S^{(r-1)}_{r}$, for $ r \geq 1$. 
Note that $|S^{(r-1)}_{r}\cap \mathrm{int}(CH(S^{(r-1)}_{r}))| =  |S_r^{(r)}|$ and, hence, 
\begin{align}\label{eq:Sr}
 |S^{(r-1)}_{r}| = |V(CH(S_r^{(r-1)})| + |S^{(r-1)}_{r}\cap \mathrm{int}(CH(S^{(r-1)}_{r}))| = |V(CH(S_r^{(r-1)})| + |S_r^{(r)}| . 
\end{align}

\begin{figure}[ht]
\centering
\begin{tikzpicture}[scale=1.05,
    point/.style={circle,inner sep=2.2pt},
    hullone/.style={red!70!black,thick},
    hulltwo/.style={blue!70!black,thick}
]


\fill[red!20,opacity=0.35]
    (0.8,1.0) --
    (1.4,4.8) --
    (4.5,5.6) --
    (7.7,4.7) --
    (8.3,1.2) --
    (4.6,0.5) --
    cycle;


\fill[blue!20,opacity=0.45]
    (2.2,2.0) --
    (3.0,4.0) --
    (5.8,4.2) --
    (6.8,2.2) --
    (4.4,1.4) --
    cycle;


\draw[hullone]
    (0.8,1.0) --
    (1.4,4.8) --
    (4.5,5.6) --
    (7.7,4.7) --
    (8.3,1.2) --
    (4.6,0.5) --
    cycle;

\draw[hulltwo]
    (2.2,2.0) --
    (3.0,4.0) --
    (5.8,4.2) --
    (6.8,2.2) --
    (4.4,1.4) --
    cycle;


\foreach \x/\y in {
    0.8/1.0,
    1.4/4.8,
    4.5/5.6,
    7.7/4.7,
    8.3/1.2,
    4.6/0.5,
    3.0/2.1
}
    \fill[red] (\x,\y) circle (2.5pt);

\foreach \x/\y in {
    2.2/2.0,
    3.0/4.0,
    5.8/4.2,
    6.8/2.2,
    4.4/1.4,
    4.5/3.0
}
    \fill[blue] (\x,\y) circle (2.5pt);

\foreach \x/\y in {
    0.5/3.4,
    8.5/3.2,
    2.0/3.5,
    2.7/3.0,
    4.0/2.4,
    5.1/2.4,
    4.7/3.4
}
    \fill[green!60!black] (\x,\y) circle (2.5pt);

\foreach \x/\y in {
    1.0/5.6,
    8.0/5.5,
    1.1/0.3,
    8.7/0.4,
    7.0/3.5,
    3.7/3.5,
    5.2/3.1,
    4.4/2.7
}
    \fill[orange] (\x,\y) circle (2.5pt);


\node[red!70!black,above left] at (7.15,5.15)
    {$CH(S_1^{(0)})$};

\node[blue!70!black] at (6.85,4.2)
    {$CH(S_2^{(1)})$};

\end{tikzpicture}
\caption{\small{Illustration of the nested sets
$S^{(0)}\supseteq S^{(1)}\supseteq S^{(2)}$.
The shaded red region is $CH(S_1^{(0)})$, whose interior determines
$S^{(1)}$. The blue region is $CH(S_2^{(1)})$, whose interior determines
$S^{(2)}$.}}
\label{fig:setcolors}
\end{figure}
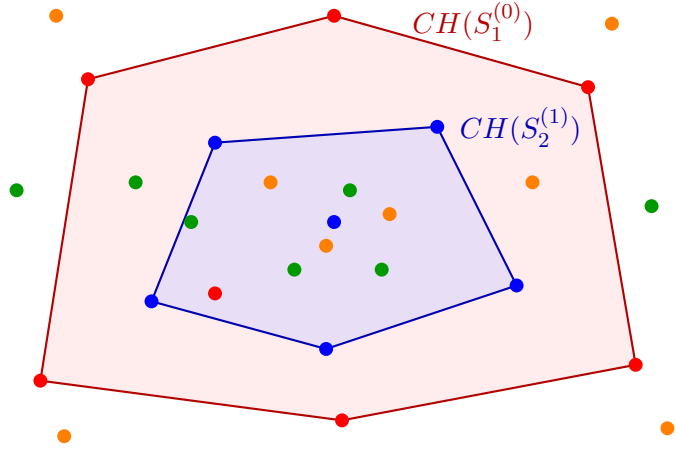

The first main ingredient in the proof is the following lemma. This is a local blocking recursion, which shows that if the chosen color class generates many monochromatic triangular faces, there must be enough differently colored points to block all of them. To this end, for $r \geq 1$, define $\gamma_r = \sum_{a=1}^{r} |S^{(r)}_a|$ and $\gamma_0=0$.

\begin{lemma}\label{lm:S}
Define $s_c = c-  t_c+1$.  For every $1\leq r \leq R$,
\begin{align}\label{eq:S}
        s_c (|S^{(r-1)}_{r}|+|S_r^{(r)}|-2)&\leq \gamma_{r-1}+(c-r)|S^{(r-1)}_{r}| . 
\end{align} 
\end{lemma}

The second key ingredient is the following lemma, which controls the cumulative contribution of previously selected colors, when the recursive triangulation is iterated for $R$ steps (recall \eqref{eq:Rlayers}).

\begin{lemma}\label{lm:SRsc} 
For every sufficiently large $c$, there exists $N_c\geq 1$ and a bad set $S$, with $|S|\geq N_c$, such that the following hold:  
\begin{align}\label{eq:SRsc}
\gamma_{R-1}=o(\varepsilon_c) |S_{R}^{(R-1)}|  \quad \text{ and } \quad 2s_c\leq |S_R^{(R-1)}| , 
\end{align}
where $s_c$ and $\varepsilon_c$ are defined in \eqref{eq:sc} and \eqref{eq:Rlayers}, respectively.   
\end{lemma}

The proofs of Lemma \ref{lm:S} and Lemma \ref{lm:SRsc} are given in Section \ref{sec:Spf} and Section \ref{sec:SRscpf}. We now apply these lemmas to complete the proof of Theorem \ref{thm:triangle}. To begin with, applying \eqref{eq:S} with $r=R$, gives  
$$s_c\bigl(|S_R^{(R-1)}|+|S_R^{(R)}|-2\bigr) \leq \gamma_{R-1}+(c-R)|S_R^{(R-1)}|.$$ 
Thus, $(s_c-c+R)|S_R^{(R-1)}| +s_c|S_R^{(R)}|-2s_c \leq \gamma_{R-1}$. Since $s_c-c+R =\varepsilon_c+1$, this implies 
\begin{equation*} 
(\varepsilon_c+1)|S_R^{(R-1)}|-2s_c \leq \gamma_{R-1}. 
\end{equation*} 
Then, applying the second inequality in \eqref{eq:SRsc} gives $\varepsilon_c |S_R^{(R-1)}| \leq \gamma_{R-1}$, which is incompatible with the first  inequality in \eqref{eq:SRsc}. This leads to a contradiction to the hypothesis that $M_3(c, c -  t_c ) =\infty$.

\subsection{Proof of Lemma \ref{lm:S}} 
\label{sec:Spf}

Fix $1\leq r \leq R$. The bound in \eqref{eq:S} holds trivially, when $|S^{(r-1)}_{r}| \leq 2$. Hence, we can assume $|S^{(r-1)}_r| \geq 3$. Note that, by Fact \ref{ft:triangle} and \eqref{eq:Sr}, any triangulation of $S^{(r-1)}_{r}$ has 
$$2 |S^{(r-1)}_{r}| - |V(CH(S_r^{(r-1)})|-2 = |S^{(r-1)}_{r}|+|S_r^{(r)}|-2$$ 
triangular faces. Note every monochromatic triangle in $S^{(r-1)}$ has at least $s_c$ points in their interiors. Also, since the regions are convex and nested, every point of $S$ lying in a triangle whose vertices belong to $S^{(r-1)}$ also belongs to $S^{(r-1)}$. Moreover, the interior of a triangular face of a triangulation of $S_r^{(r-1)}$ contains no point of $S_r^{(r-1)}$. Thus, all of its interior points have colors different from $r$. Hence, 
\begin{align}\label{eq:Sra}
\sum_{a \in [c] \backslash\{r\}} |S_a^{(r-1)}| \geq s_c(|S^{(r-1)}_{r}|+|S_r^{(r)}|-2)  .  
\end{align}
Now, note that 
$$\sum_{a \in [c] \backslash\{r\}} |S_a^{(r-1)}| = \sum_{a =1}^{r-1} |S_a^{(r-1)}| + \sum_{a \in [c] \backslash [r] } |S_a^{(r-1)}| \leq \gamma_{r-1} + (c-r) |S_r^{(r-1)}|,$$
by the definition of $\gamma_{r-1}$ and using $|S_a^{(r-1)}| \leq |S_r^{(r-1)}|$, for $a \in [c] \backslash [r]$. Combining the above inequalities gives the result  in \eqref{eq:S}.  \hfill $\Box$  

\subsection{Proof of Lemma \ref{lm:SRsc}} 
\label{sec:SRscpf}

For $1\leq r \leq t<R$ define 
\begin{align}\label{eq:brt}
b_{r,t}=\prod_{a=r}^{t}\frac{c-a}{s_c}  ,  
\end{align}
and the empty product $b_{r, r-1}=1$. We first obtain a uniform bound $b_{r,t}$ (recall \eqref{eq:brt}).

\begin{observation}\label{observation:brt} 
For all sufficiently large $c$, $$ b_{r,t} = O\left(\frac{\sqrt{c}}{\log^2 c}\right) = o(\varepsilon_c), $$ uniformly over $1\leq r\leq t<R$. 
\end{observation} 

\begin{proof} If $a\geq  t_c-1$, then $\frac{c-a}{s_c}\leq 1$. Also, if $a\leq t_c-2$, then 
$\frac{c-a}{s_c} = 1+\frac{ t_c-1-a}{s_c}$.  Hence, recalling \eqref{eq:brt} and using $\log(1+x)\leq x$, 
\begin{align*} 
\log b_{r,t} \leq \sum_{a=1}^{ t_c-1} \log\left(1+\frac{a}{s_c}\right) \leq \frac{1}{s_c}\sum_{a=1}^{ t_c-1}a \leq \frac{ t_c^2}{2s_c}. 
\end{align*} 
Recalling \eqref{eq:interiorpoints} and \eqref{eq:sc}, it follows that 
\begin{align}\label{eq:interiorpointssc}
\frac{t_c^2}{2s_c} \leq \frac{c\log c-4c\log\log c}{2\bigl(c-\sqrt{c\log c-4c\log\log c}+1\bigr)} =
\frac12\log c-2\log\log c+o(1)  .    
\end{align}  
Combining the above, the observation follows. 
\end{proof}

To prove Lemma \ref{lm:SRsc} we need to bound $\gamma_{R-1}$. To this end, recalling \eqref{eq:sequence} first note that 
\begin{align}\label{eq:Srsize}
\gamma_{R-1} = \sum_{r=1}^{R-1}|S_r^{(R-1)}| \leq \sum_{r=1}^{R-1}|S_r^{(r)}| . 
\end{align}
Recalling \eqref{eq:sc} and arranging the terms in \eqref{eq:S} gives, for every $1 \leq r<R$, 
$$s_c|S_r^{(r)}| \leq \gamma_{r-1} +( t_c-1-r)|S_r^{(r-1)}| +2s_c.$$ 
For $r\geq t_c-1$, the second term in  the RHS above is nonpositive. Therefore, summing over $1\leq r<R$, we obtain 
\begin{align} 
\sum_{r=1}^{R-1}|S_r^{(r)}| & \leq \frac{1}{s_c} \sum_{r=1}^{R-1}\gamma_{r-1} + \frac{1}{s_c} \sum_{r=1}^{ t_c-2} ( t_c-1-r)|S_r^{(r-1)}| +2(R-1) \nonumber \\ 
& \leq \frac{R}{s_c} \sum_{a=1}^{R-1}|S_a^{(a)}| +  \frac{1}{s_c} \sum_{r=1}^{ t_c-2} ( t_c-1-r)|S_r^{(r-1)}| +2(R-1) , 
\label{eq:srb}
\end{align} 
since, for every $r<R$ we have $$ \gamma_{r-1} = \sum_{a=1}^{r-1}|S_a^{(r-1)}| \leq \sum_{a=1}^{r-1}|S_a^{(a)}| \leq \sum_{a=1}^{R-1}|S_a^{(a)}|. $$ 
Hence, to bound the RHS of \eqref{eq:srb} we need to bound $|S_r^{(r-1)}|$. 

\begin{observation} The following holds, uniformly over $ 1 \leq r < R$,
\begin{equation}\label{eq:srsa} 
|S_r^{(r-1)}| \leq b_{r, R-1}|S_R^{(R-1)}| + O\left( \frac{R \sqrt{c} }{s_c \log^2 c} \sum_{a=1}^{R-1}|S_a^{(a)}| \right) +O\left(\frac{ R \sqrt{c}}{ \log^2 c}\right). 
\end{equation}  
\end{observation}

\begin{proof}  Fix $1\leq r < R$. Note that the number of points inside $CH(S_r^{(r-1)})$ with colors in  $[c]\backslash [r]$ is at least $$s_c(|S^{(r-1)}_{r}|+|S_r^{(r)}|-2) - \gamma_{r-1} . $$  This is because there are at least $s_c(|S^{(r-1)}_{r}|+|S_r^{(r)}|-2)$ points inside $CH(S_r^{(r-1)})$ of color different from $r$ (by arguments similar to \eqref{eq:Sra}) and points with colors in $[r-1]$ have contribution at most $\gamma_{r-1}$. Since there are $c-r$ unused colors, the most frequent unused color in $S^{(r)}$ satisfies: 
\begin{align}\label{eq:Srlayer}
|S_{r+1}^{(r)}| \geq \frac{s_c (|S^{(r-1)}_{r}|+|S_r^{(r)}|-2)-\gamma_{r-1}}{c-r} \geq \frac{s_c (|S^{(r-1)}_{r}|-\gamma_{r-1}-2 s_c }{c-r} , 
\end{align} 
where the second inequality in \eqref{eq:Srlayer} follows by dropping the term $s_c|S_r^{(r)}|$. Now, rearranging the terms in \eqref{eq:Srlayer} gives, 
$$|S_r^{(r-1)}| \leq \frac{c-r}{s_c}|S_{r+1}^{(r)}| +\frac{\gamma_{r-1}}{s_c}+2.$$ Iterating this inequality backwards from $R$ to $r$ gives \begin{equation}\label{eq:backward-limit} |S_r^{(r-1)}| \leq b_{r, R-1}|S_R^{(R-1)}| +\frac{1}{s_c}\sum_{t=r}^{R-1}b_{r,t-1}\gamma_{t-1} +2\sum_{t=r}^{R-1}b_{r,t-1}, 
\end{equation} 
where an empty product is understood to be equal to $1$.  By Observation \ref{observation:brt} and \eqref{eq:backward-limit}, the result in \eqref{eq:srsa} follows.  
\end{proof}

Note that the first term in \eqref{eq:srb} is at most $\frac{R}{s_c} \sum_{a=1}^{R-1}|S_a^{(a)}| = o( \sum_{a=1}^{R-1}|S_a^{(a)}|)$, since $R=O(\sqrt{c\log c})$ and $s_c\sim c$. Using this and \eqref{eq:srsa} in \eqref{eq:srb} gives, 
\begin{align}\label{eq:Srr} 
\sum_{r=1}^{R-1}|S_r^{(r)}| & \leq o\left( \sum_{a=1}^{R-1}|S_a^{(a)}| \right) + \left( \frac{1}{s_c} \sum_{r=1}^{ t_c-2} ( t_c-1-r)b_{r, R-1} \right) |S_R^{(R-1)}| \nonumber \\ 
& \hspace{2cm} + O\left( \frac{ R  t_c^2 \sqrt{c} }{s_c^2 \log^2 c } \sum_{a=1}^{R-1}|S_a^{(a)}| \right) + O\left( \frac{ R  t_c^2 \sqrt{c}}{ s_c \log^2 c} +R \right) \nonumber \\ 
& \leq o\left( \sum_{a=1}^{R-1}|S_a^{(a)}| \right) + \left( \frac{1}{s_c} \sum_{r=1}^{ t_c-2} ( t_c-1-r) b_{r, R-1} \right) |S_R^{(R-1)}| + O\left( \frac{ R  t_c^2 \sqrt{c}}{ s_c \log^2 c} +R \right) , 
\end{align}
since $\frac{ R  t_c^2 \sqrt{c} }{s_c^2 \log^2 c } =o(1)$ (recall \eqref{eq:interiorpoints}, \eqref{eq:sc}, and \eqref{eq:Rlayers}).  
We now bound the term $b_{r, R-1}$. For this, recalling \eqref{eq:brt}, note that 
\begin{align*} 
b_{r, R-1} \leq \prod_{t=r}^{ t_c-2} \left( 1+\frac{ t_c-1-t}{s_c} \right) \leq \exp\left( \frac{( t_c-1-r)( t_c-r)}{2s_c} \right).
\end{align*}
Hence, recalling \eqref{eq:interiorpointssc}, 
\begin{align}\label{eq:exponentialtc}
\frac{1}{s_c} \sum_{r=1}^{ t_c-2} ( t_c-1-r)b_{r, R-1} \leq \frac{1}{s_c} \sum_{a=1}^{ t_c-2} a \exp\left(\frac{a(a+1)}{2s_c}\right) \leq \frac{ t_c^2}{s_c} \exp\left(\frac{ t_c^2}{2s_c}\right) = o(\varepsilon_c). 
\end{align} 
Therefore, \eqref{eq:Srr} gives, 
\begin{align}\label{eq:diagonal-small} 
\sum_{r=1}^{R-1}|S_r^{(r)}| & \leq o(\varepsilon_c)|S_R^{(R-1)}| + o\left( \sum_{a=1}^{R-1}|S_a^{(a)}| \right)+ O\left( \frac{ R  t_c^2 \sqrt{c}}{ s_c \log^2 c} +R \right) \nonumber \\ 
& \leq (1+o(1) ) \left( o(\varepsilon_c)|S_R^{(R-1)}| + O\left( \frac{ R  t_c^2 \sqrt{c}}{ s_c \log^2 c} +R \right) \right) , 
\end{align} 
by absorbing the $o( \sum_{a=1}^{R-1}|S_a^{(a)}|)$ to the LHS. Next, we show that the additive term above is negligible.  From \eqref{eq:srsa}, taking $r=1$, we have 
$$|S_1^{(0)}| \leq b_{1, R-1}|S_R^{(R-1)}| + O\left( \frac{R \sqrt{c} }{s_c \log^2 c} \sum_{a=1}^{R-1}|S_a^{(a)}| \right) + O\left(\frac{ R \sqrt{c}}{ \log^2 c}\right).$$ 
On the other hand, \eqref{eq:diagonal-small} gives 
$$\sum_{a=1}^{R-1}|S_a^{(a)}| \leq o(\varepsilon_c)|S_R^{(R-1)}| + O\left( \frac{R  t_c^2 \sqrt{c}}{ s_c \log^2 c } +R \right).$$ 
Substituting this into the preceding inequality shows that there are constants $K > 0 $ and $K' > 0$ (depending only on $c$) such that 
 $$|S_1^{(0)}| \leq K |S_R^{(R-1)}| + K'.$$ 
On the other hand, color $1$ is chosen to be a  largest color class in $S$, which means, $|S_1^{(0)}|\geq \frac{|S|}{c}$.  The above inequalities combined show that 
$$|S_R^{(R-1)}| \geq  \frac{1}{K c} |S| - \frac{K'}{K}  .  $$ 
Hence, we may choose $N_c$ sufficiently large and then choose a bad set $S$ with $|S|>N_c$ so that $|S_R^{(R-1)}|$ is large enough such that:  
$$O\left( \frac{ R  t_c^2 \sqrt{c}}{ s_c \log^2 c} +R \right) \leq |S_R^{(R-1)}| \quad \text{ and } \quad 2s_c\leq |S_R^{(R-1)}|  ,  $$  
which proves the second inequality in \eqref{eq:SRsc}. Further, the first inequality above together with \eqref{eq:diagonal-small} gives 
$\sum_{r=1}^{R-1}|S_r^{(r)}| = o(\varepsilon_c)|S_R^{(R-1)}|$. This, together with \eqref{eq:Srsize}, proves the first condition in \eqref{eq:SRsc}. This completes the proof of Lemma \ref{lm:SRsc}. \hfill $\Box$

\begin{remark} 
\label{remark:product}
{\em The product in \eqref{eq:brt} admits a simple birthday-paradox interpretation (see, for example, \cite[Section~7.1]{DiaconisMosteller1989} and \cite{birthday}). Indeed,
\begin{align}\label{eq:b}
b_{r,t} = \left(\frac{c}{s_c}\right)^{t-r+1} \prod_{a=r}^{t}\left(1-\frac{a}{c}\right).
\end{align}
The second factor is the probability that, in the usual birthday model with
$c$ equally likely values, the next $t-r+1$ draws produce no repetition,
conditional on having already seen $r$ distinct values. Thus the products
$b_{r,t}$ may be viewed as a birthday no-collision probability multiplied by
the deterministic factor $(\frac{c}{s_c})^{t-r+1}$. In the relevant scale $t=O(\sqrt{c\log c})$, the product in \eqref{eq:b} has  Gaussian-type growth (up to lower order terms), 
$$\exp\left(\frac{t^2}{2c}\right).$$
This gives a heuristic explanation for the appearance of the $\sqrt{c\log c}$ scale in the proof, where  obtaining a polynomial gain in $c$ requires $\frac{t^2}{c}$ to be of order $\log c$.   The products $b_{r,t}$ also appear if the colors are exposed in a uniformly random order, with the same convex-hull restriction performed after each step. Conditional on the first $a$ step, the expected size of the next color class is the total surviving population of the unused colors divided by $c-a$. Combining this identity with the triangulation blocking inequality yields the factor $\frac{c-a}{s_c}$ in the backward recursion.  } 
\end{remark}

%
%

\section{ Proof of Theorem \ref{thm:colortriangle} }
\label{sec:colortrianglepf}

Let $A$ and $B$ be two finite point sets in the plane with pairwise distinct $x$-coordinates. We say that $A$ lies \emph{high above} $B$ if every point of $A$ lies above every line determined by two points of $B$, and every point of $B$ lies below every line determined by two points of $A$.
Using this notion, Horton sets are defined recursively \cite{horton1983sets} as follows (see also \cite[Chapter 3]{discretegeometry}).

\begin{definition}\label{defn:H}
{\em  A finite point set with distinct
$x$-coordinates $H=\{p_1,p_2,\ldots,p_n\}$, labeled in increasing order of their $x$-coordinates, is called a \emph{Horton set} if the following conditions hold recursively:
\begin{itemize}
    \item if $|H|\leq 2$, then $H$ is a Horton set, 
    \item if $|H|\geq 3$, then the sets
    $$H^{-} :=  \{p_1,p_3,p_5,\ldots\} \quad\text{ and } \quad H^{+} := \{p_2,p_4,p_6,\ldots\}$$ are both Horton sets, and $H^{+}$ lies high above $H^{-}$. The sets $H^{+}$ and $H^{-}$ will be referred to as the \emph{upper child} and \emph{lower child} of $H$, respectively, and $H$ will be referred to as their \emph{parent}.
\end{itemize} 
}   
\end{definition}

The proof of Theorem \ref{thm:colortriangle} proceeds by induction on $c$. The case $c=1$ is immediate: any triangulation of a set of at least three points contains a triangular face whose interior contains no point of the
set.  Thus, we may take $n_1=3$. Hence, suppose that $c\geq 2$ and that the result in Theorem \ref{thm:colortriangle} holds for $c-1$.  Define $K_c:= \lfloor\log_2(2c) \rfloor$, and choose $n_c$ sufficiently large such that 
$$n_c\ge 2c+1 \quad\text{ and } \quad \left\lfloor\frac{n_c}{2^{K_c}}\right\rfloor\ge n_{c-1}.$$ Consider a  Horton set $H=\{h_1,\ldots,h_n\}$ with $n\geq n_c$, indexed in increasing order of $x$-coordinates and colored with at most $c$ colors. Note that for any $2c+1$ consecutive points of $H$, by the pigeonhole principle, three of them have the same color.  Also, two of these three indices have the same parity.  Denote the corresponding points by $h_u$ and $h_v$, where $u < v$.  Then
\begin{align}\label{eq:uvc}
2\mid (v - u )  \quad  \text{ and } \quad v - u  \leq 2c. 
\end{align}
Write $v-u=2^s t$, where $s \geq 1$ and $t$ is odd.  Then, by \eqref{eq:uvc}, $s \leq K_c$ and $t \leq c$.
Since $2^s\mid(v-u)$, the points $h_u$ and $h_v$ belong to the same recursive Horton subset $A$ at depth $s$.  Moreover, their positions in the $x$-ordering of $A$ differ by $\frac{1}{2^s} (v-u) = t$. Also, since $t$ is odd, $h_u$ and $h_v$ belong to different children of $A$. Let $B$ be the parent of $A$, and let $A'$ be the other child of $B$.  We now consider the following two cases:

\begin{itemize}

\item There exists $z \in A'$ which has the same color as the points $h_u$ and $h_v$. Then by   Lemma~\ref{lm:horton-triangle},
$$\left| \mathrm{int}(\triangle h_u h_v z)\cap H \right| \leq \frac{t-1}{2} \leq \left\lfloor\frac{c-1}{2}\right\rfloor, $$
since $t$ is odd and $t \leq c$. Hence, $\triangle h_u h_v z$ is a monochromatic triangle with the required bound on the number of interior points.

\item Suppose, on the other hand, that $A'$ contains no point of this color.
Then $A'$ is colored with at most $c-1$ colors.  Since $A'$ occurs at depth $s \leq K_c$ in the recursive decomposition of $H$,
$$ |A'| \geq \left\lfloor\frac{|H|}{2^s}\right\rfloor  \geq \left\lfloor\frac{n_c}{2^{K_c}}\right\rfloor \geq n_{c-1}.$$
Hence, the induction hypothesis applied to the Horton set $A'$, gives a
monochromatic triangle $\triangle$ satisfying 
$$\left| \mathrm{int}(\triangle )\cap A' \right| \leq \left\lfloor\frac{c-2}{2}\right\rfloor \leq \left\lfloor\frac{c-1}{2}\right\rfloor  ,  $$
Also, since no point of $H\setminus A'$ lies in $CH(A')$, $
\mathrm{int}(\triangle)\cap H = \mathrm{int}(\triangle)\cap A'$. Hence, the same monochromatic triangle has at most the required number of interior points in $H$.
This completes the induction.  \hfill $\Box$  
\end{itemize}

To complete the proof of Theorem \ref{thm:colortriangle} we need the following lemma: 

\begin{lemma}\label{lm:horton-triangle}
Let $H$ be a Horton set, and let $B$ be a set arising at some stage of the
recursive decomposition of $H$.  Let $A, A'$ be the two children of $B$.  Suppose that $p, q \in A$ belong to different children of $A$, and that their positions in the $x$-ordering of $A$ differ
by an odd integer $t$.  Then, for every $z\in A'$,
$$\left| \mathrm{int}(\triangle p q z)\cap H\right| \leq \frac{t-1}{2}.$$
\end{lemma}

\begin{proof}[Proof of Lemma \ref{lm:horton-triangle}]
We will use the following two standard properties of Horton sets.  

\begin{itemize}

\item Every set arising in the recursive decomposition of a Horton set is an {\it island}, that is, if $B$ is such a set, then $CH(B)\cap H=B$. This is because, the two children of $B$ are separated by the defining high-above property of Horton sets, and the assertion follows inductively.

\item If $B^+$ and $B^-$ are the upper and lower children of $B$, respectively, and $u\in B^+$ and
$v\in B^-$, then every point of $B^+$ whose $x$-coordinate lies strictly
between those of $u$ and $v$ lies above the line $uv$, while every such
point of $B^-$ lies below the line $uv$ (see, for example, \cite[Proposition 4.1]{colorempty}).  
\end{itemize}  
Now, without loss of generality assume $A$ is the lower child of $B$ and $A'$ is the upper
child.  Since $p, q\in A$, the point $z$ lies above the line $pq$.

\begin{observation} Every point of $H$ in the interior of $\triangle pqz$ belongs to $A$,
has $x$-coordinate strictly between those of $p$ and $q$, and lies above
the line $pq$.
\end{observation}

\begin{proof}
Since $\triangle pqz\subseteq CH(B)$ and $B$ is an island, no point of $H\setminus B$ lies in the triangle.  Moreover, no point $w\in A'\setminus\{z\}$ can lie in its interior.  Indeed, if it did, then
the line $zw$ would intersect the segment $pq$ in its relative interior, and $p$ and $q$ would lie on opposite sides of $zw$.  This contradicts the Horton-set property, since $z, w\in A'$ and all points of $A$ lie on the same side of every line through two points of $A'$.

It remains to consider points of $A$.  Assume, without loss of generality, that $p <_x q$. (Here, $<_x$ denotes the ordering with respect to the $x$-coordinates.) Every interior point of $\triangle pqz$ lies on the same side of the line $pq$ as $z$, and hence, lies above $pq$.  Also, every point of $H$ in the interior of $\triangle pqz$ has $x$-coordinate between those of $p$ and $q$.
Note that, if not, there exists $w\in A$ in the interior of $\triangle pqz$ such that $w<_x p$.  Since $w\in\operatorname{int}(\triangle pqz)$ and $w<_x p<_x q$, the $x$-coordinate of $w$ must lie strictly between the minimum and maximum $x$-coordinates of the three vertices. Hence, $z<_x w<_x p<_x q$. By the above/below property, $w$ lies below the line $zp$.  On the other hand, since $z$ lies above the line $pq$, the point $q$ lies above the line $zp$.  Thus, $w$ and the interior of $\triangle pqz$ lie on
opposite sides of the edge $zp$, which is a contradiction.  The case $q <_x w$ can be proved similarly.  
\end{proof}

Now, consider the recursive decomposition $A=A^+\sqcup A^-$. By assumption, $p$ and $q$ belong to different children of $A$.  Since their positions in the $x$-ordering of $A$ differ by an odd integer $t$, there are exactly $t-1$ points of $A$ strictly between them, with $\frac{t-1}{2}$ in each of $A^+$ and $A^-$.  Also, by the above/below property all intermediate points in one child lie above $pq$ and all intermediate points in the other child lie below $pq$.  Hence, at most $\frac{(t-1)}{2}$ points of $A$ can lie in the interior of $\triangle pqz$.
\end{proof}

\small

\subsection*{AI Declaration }  
ChatGPT 5.6 Sol was used to assist with and verify computations. The authors assume responsibility for all content.

\bibliographystyle{abbrvnat} 
\bibliography{ref.bib}

\end{document}